\documentclass[a4paper,12pt]{article}

\usepackage{amsmath,amsfonts,graphicx}

\newenvironment{proof}{\noindent {\bf Proof:}}{\hfill $\Box$}
\newtheorem{theorem}{Theorem}[section]

\newtheorem{definition}{Definition}[section]
\newtheorem{assumption}{Assumption}[section]

\def\bx{\mathbf{x}}
\def\br{\mathbf{r}}
\def\balpha{\mathbf{\alpha}}
\def\by{\mathbf{y}}
\def\bz{\mathbf{z}}

\def\R{\mathbb{R}}

\def\N{\mathbb{N}}
\def\K{\mathbf{K}}

\def\M{\mathbf{M}}

\def\z{\mathbf{z}}

\def\bx{\mathbf{x}}
\def\y{\mathbf{y}}

\begin{document}

\title{Minimizing the sum of many rational functions\thanks{The first author acknowledges
support by the French National Research Agency (ANR) through COSINUS program (project ID4CS ANR-09-COSI-005).
The second author acknowledges support by
Research Program MSM6840770038 of the Czech Ministry of Education and Project 103/10/0628 of
the Grant Agency of the Czech Republic.}}

\author{Florian Bugarin,$^{1,2}$ Didier Henrion,$^{2,3}$ Jean-Bernard Lasserre$^{2,4}$}

\footnotetext[1]{Universit\'e de Toulouse; Mines Albi;
Institut Cl\'ement Ader (ICA); Campus Jarlard, F-81013 Albi, France}

\footnotetext[2]{CNRS; LAAS; 7 avenue du colonel Roche, F-31077 Toulouse,
France; Universit\'e de Toulouse; UPS, INSA, INP, ISAE; LAAS;
F-31077 Toulouse, France}

\footnotetext[3]{Faculty of Electrical Engineering, Czech Technical
University in Prague, Technick\'a 4, CZ-16607 Prague, Czech Republic}

\footnotetext[4]{Institut de Math\'ematiques de Toulouse (IMT),
Universit\'e de Toulouse, France}

\date{\today}

\maketitle

\begin{abstract}
We consider the problem of globally minimizing the sum of many rational functions
over a given compact semialgebraic set. The number of terms can be large (10 to 100),
the degree of each term should be small (up to 10), and the number of variables can be large
(10 to 100) provided some kind of sparsity is present. We describe a formulation of the
rational optimization problem
as a generalized moment problem and its hierarchy of convex semidefinite relaxations. Under
some conditions we prove that the sequence of optimal values converges
to the globally optimal value. We show how public-domain software can be
used to model and solve such problems.
\end{abstract}

\begin{center}\footnotesize
\end{center}

\section{Introduction}

Consider the optimization problem
\begin{equation}
\label{defpb}
f^*\,:=\,\displaystyle\inf_{\bx\in\K}\:\sum_{i=1}^N f_i(\bx)
\end{equation}
over the basic semi-algebraic set
\begin{equation}
\label{setk}
\K\,:=\,\{\bx\in\R^n\::\: g_j(\bx)\,\geq\,0,\:j=1,\ldots,m\:\},
\end{equation}
for given polynomials $g_j\in\R[\bx]$, $j=1,\ldots,m$,
and where each term $f_i:\R^n\to\R$ is a rational function
\[
\bx\mapsto f_i(\bx) := \frac{p_i(\bx)}{q_i(\bx)},
\]
with $p_i,q_i\in\R[\bx]$ and $q_i>0$ on $\K$,
for each $i=1,\ldots,N$.

Problem (\ref{defpb}) is a {\it fractional programming} problem of a rather general form. Nevertheless,
we assume that the degree of each $f_i$ and $g_j$ is relatively small (up to 10),
but the number of terms $N$ can be quite large (10 to 100). For dense data
the number of variables $n$ should also be small (up to 10). However,
this number can be also quite large (10 to 100) provided that the problem data feature
some kind of sparsity (to be specified later). Even though problem (\ref{defpb})
is of self-interest, our initial motivation came from some applications in computer vision,
where such problems are typical. These applications will be described elsewhere.

In such a situation, fractional programming problem (\ref{defpb}) is quite challenging.
Indeed,  we make {\it no} assumption on the polynomials $p_i,q_i$ whereas
even with a relatively small number of fractions and under convexity (resp. concavity) assumptions on 
$p_i$ (resp. $q_i$), problem (\ref{defpb})
is hard to solve (especially if one wants to compute the global minimum);
see for example the survey \cite{schaible}
and references therein.

We are interested in solving problem (\ref{defpb}) globally, in the sense
that we do not content ourselves with a local optimum satisfying first order
optimality conditions, as typically obtained with standard local optimization
algorithms such as Newton's method or its variants. If problem (\ref{defpb})
is too difficult to solve globally (because of ill-conditioning and/or
too large a number of variables or terms in the objective function), we
would like to have at least a valid lower bound on the global minimum,
since upper bounds can be obtained with local optimization algorithms.

One possible approach is to reduce all fractions $p_i/q_i$ 
to same denominator and obtain a single
rational fraction to minimize. Then one may try to apply the hierarchy
of semidefinite programming (SDP) relaxations defined in \cite{deklerk},
see also \cite[Section 5.8]{book}. But such a strategy is not appropriate
because the degree of the common denominator is potentially large and even if $n$ is small,
one may not even implement the first relaxation of the hierarchy, due to the
present limitations of SDP solvers. Moreover, in general this strategy also destroys
potential sparsity patterns present in the original formulation (\ref{defpb}),
and so precludes from using an appropriate version (for the rational fraction case) 
of the sparse semidefinite relaxations introduced in \cite{waki} whose convergence
was proved in \cite{sparse} under some conditions on the sparsity pattern, 
see also \cite[Sections 4.6 and 5.3.4]{book}.

Another possibility is to introduce additional variables $r_i$
(that we may call liftings) with associated constraints
\[\frac{p_i(\bx)}{q_i(\bx)} \leq r_i, \:i=1,\ldots,N,\]
and solve the equivalent problem:
\begin{equation}
\label{defpb1}
f^*\,:=\,\displaystyle\inf_{(\bx,\br)\in\widehat{\K}}\:\sum_{i=1}^N r_i
\end{equation}
which is now a {\it polynomial} optimization problem in the new variables
$(\bx,\br)\in\R^n\times\R^N$, and where the new feasible set $\widehat{\K}
= \K \times \{ (\bx,\br) \in\R^{n+N}  \: :\: r_i q_i(\bx) - p_i(\bx) \geq 0 \}$
is modeling the epigraphs of the rational terms. The sparsity pattern is preserved
and if $\K$ is compact one may in general obtain upper and lower bounds 
$\overline{r}_i,\underline{r}_i$ on the $r_i$ so as
to make $\widehat{\K}$ compact by adding the quadratic (redundant) constraints
$(r_i-\underline{r}_i)(\overline{r}_i-r_i)\geq0$, $i=1,\ldots,N$, 
and apply the sparse semidefinite relaxations. However, in doing so one introduces
$N$ additional variables, and this may have an impact on the overall
performance, especially if $N$ is large. In the sequel this approach is referred to
as the {\it epigraph approach}.

The goal of the present paper is to circumvent all above difficulties
in the following two situations: either $n$ is relatively small, or
$n$ is potentially large but some sparsity is present, i.e.,
each $f_i$ and each $g_j$ in (\ref{defpb}) is concerned with only a small subset
of variables. In the approach that we propose, we do not need
the epigraph liftings.
The idea is to formulate (\ref{defpb}) as an equivalent infinite-dimensional linear
problem which a particular instance of the generalized moment problem (GMP)
as defined in \cite{gmp}, with $N$ unknown measures (where each measure is associated with 
a fraction $p_i/q_i$).
In turn this problem can be easily modeled and solved
with our public-domain software GloptiPoly 3 \cite{g3}, a significant update
of GloptiPoly 2 \cite{g2}. In the sequel this approach
is referred to as the {\it GMP approach}.

The outline of the paper is as follows. In Section \ref{dense} we introduce the
SDP relaxations first in the case that $n$ is small and the data are dense polynomials.
Then in Section \ref{sparse} we extend the SDP relaxations to the case that $n$ is
large but sparsity is present. In Section \ref{gmp} we show how the
GMP formulation can be exploited to model the SDP relaxations
of problem (\ref{defpb}) easily with GloptiPoly 3. We also provide
a collection of numerical experiments showing the relevance of our GMP approach,
especially in comparison with the epigraph approach.

\section{Dense SDP relaxations}\label{dense}

In this section we assume that $n$, the number of variables in problem (\ref{defpb}),
is small, say up to 10.

\subsection{GMP formulation}

Consider the infinite dimensional linear problem
\begin{equation}
\label{lp0}
\begin{array}{rl}
\hat{f}\,:=\,\displaystyle\inf_{\mu_i\in\mathcal{M}(\K)}&
\displaystyle\sum_{i=1}^N\int_{\K}p_i\,d\mu_i\\
\mbox{s.t.}
&\displaystyle\int_{\K}q_1d\mu_1=1\\
&\\
&\displaystyle\int_{\K}\bx^{\balpha}q_i d\mu_i=
\displaystyle\int_{\K}\bx^{\balpha}q_1 d\mu_1,
\quad \forall\balpha\in\N^n,\: i=2,\ldots,N,
\end{array}
\end{equation}
where $\mathcal{M}(\K)$ is the space of finite Borel measures
supported on $\K$.
\begin{theorem}
\label{th1}
Let $\K\subset\R^n$ in (\ref{setk}) be compact, and assume that $q_i>0$ on $\K$, $i=1,\ldots,N$.
Then $\hat{f}=f^*$.
\end{theorem}
\begin{proof}
We first prove that $f^*\geq\hat{f}$. As $f=\sum_i p_i/q_i$ is continuous on $\K$,
there exists a global minimizer $\bx^*\in\K$ with $f(\bx^*)=f^*$. 
Define $\mu_i:=q_i(\bx^*)^{-1}\delta_{\bx^*}$, $i=1,\ldots,N$, where
$\delta_{\bx^*}$ is the Dirac measure at $\bx^*$. Then obviously,
the measures $(\mu_i)$, $i=1,\ldots,N$, are feasible for (\ref{lp0}) with associated value
\[\sum_{i=1}^N \int_\K p_i d\mu_i\,=\,\sum_{i=1}^N
p_i(\bx^*)/q_i(\bx^*)=f(\bx^*)=f^*.\] 

Conversely, let $(\mu_i)$ be a feasible solution of (\ref{lp0}).
For every $i=1,\ldots,N$, let $d\nu_i$ be the measure $q_id\mu_i$, i.e.
\[\nu_i(B):=\int_{\K\cap B} q_i(\bx)d\mu_i(\bx)\]
for all sets $B$ in the Borel $\sigma$-algebra of $\R^n$, and so the support of $\nu_i$ is $\K$.
As measures on compact sets are moment determinate, the moments constraints of
(\ref{lp0}) imply that $\nu_i=\nu_1$, for every $i=2,\ldots,N$,
and from $\int_\K q_1d\mu_1=1$ we also deduce that $\nu_1$ is a probability measure on $\K$. 
But then
\begin{eqnarray*}
\sum_{i=1}^N\int_\K p_i d\mu_i&=&\sum_{i=1}^N\int_\K \frac{p_i}{q_i}q_i d\mu_i
=\sum_{i=1}^N\int_\K \frac{p_i}{q_i}d\nu_1\\
&=&\int_\K\left(\sum_{i=1}^N \frac{p_i}{q_i}\right)\,d\nu_1\,=\,
\int_\K fd\nu_1\geq \int_\K f^*d\nu_1\,=\,f^*,
\end{eqnarray*}
where we have used that $f\geq f^*$ on $\K$ and $\nu_1$ is a probability measure on $\K$.
\end{proof}

We next make the following assumption meaning that set $\K$ admits an
algebraic certificate of compactness.
\begin{assumption}
\label{ass1}
The set $\K\subset\R^n$ in (\ref{setk}) is compact and the quadratic polynomial
$\bx\mapsto M-\Vert\bx\Vert^2$ can be written as
\[M-\Vert\bx\Vert^2=\sigma_0+\sum_{j=1}^m\sigma_j\,g_j,\]
for some polynomials $\sigma_j\in\R[\bx]$, all sums of squares of polynomials.
\end{assumption}

\subsection{A hierarchy of dense SDP relaxations}

Let $\y_i=(y_{i\balpha})$ be a real sequence indexed in the canonical basis 
$(\bx^{\balpha})$ of $\R[\bx]$, $i=1,\ldots,N$, and for every $k\in\N$,  let
$\N^n_k:=\{\alpha\in\N^n:\sum_j\alpha_j\leq k\}$.

Define the {\it moment} matrix $M_k(\y_i)$ of order $k$, associated with $\y$, whose entries
indexed by multi-indices $\beta$ (rows) and $\gamma$ (columns) read
\[
[M_k(\y_i)]_{\beta,\gamma} := y_{i(\beta+\gamma)}, \quad \forall\,\beta,\,\gamma\in\N^n_k,
\]
and so are linear in $\y_i$. Similarly, given a polynomial $g(\bx)=\sum_{\alpha} g_{\alpha} \bx^{\alpha}$,
define the {\it localising} matrix $M_k(g\y_i)$ of order $k$, associated with $\y$ and $g$, whose entries read
\[
[M_k(g\,\y_i)]_{\beta,\gamma} := \sum_{\alpha} g_{\alpha} y_{i(\alpha+\beta+\gamma)},
\quad\forall\,\beta,\gamma\in\N^n_k.
\]
In particular, matrix $M_0(g\y_i)$ is identical to $L_{\y_i}(g)$ where for every $i$,
$L_{\y_i}:\R[\bx]\to\R$ is the linear functional defined by:
\[
g\:\mapsto\:L_{\y_i}(g)\,:=\,\sum_{\alpha\in\N^n}g_\alpha y_{i\alpha},\qquad\forall g\in\R[\bx].
\]
Let $u_i:=\lceil ({\rm deg}\,q_i)/2\rceil$, $i=1,\ldots,N$,
$r_j:=\lceil ({\rm deg}\,g_j)/2\rceil$, $j=1,\ldots,m$,
and with no loss of generality assume that $u_1\leq u_2\leq \ldots\leq u_N$.
Consider the hierarchy of semidefinite programming (SDP) relaxations:
\begin{equation}
\label{hierarchy1}
\begin{array}{rll}
f^*_k=\displaystyle\inf_{\y_i} &\displaystyle\sum_{i=1}^N L_{\y_i}(p_i)&\\
\mbox{s.t.}& M_k(\y_i)\succeq0,&i=1,\ldots,N\\
&M_{k-r_j}(g_j\y_i)\succeq0,& i=1,\ldots,N,\:j=1,\ldots,m\\
&L_{\y_1}(q_1)\,=\,1\\
&L_{\y_i}(\bx^{\balpha}q_i)=L_{\y_1}(\bx^{\balpha}q_1),&\forall\balpha\in
\N^n_{2(k-u_i)},\, i=2,\ldots,N.
\end{array}\end{equation}

\begin{theorem}\label{th2}
Let Assumption \ref{ass1} hold and
consider the hierarchy of SDP relaxations (\ref{hierarchy1}). Then it follows that
\begin{enumerate}
\item[(a)] $f^*_k\uparrow f^*$ as $k\to\infty$. 
\item[(b)] Moreover, if $(\y_i^k)$ is an optimal solution of (\ref{hierarchy1}), 
and if
\[{\rm rank}\,M_k(\y_i^k)\,=\,{\rm rank}\,M_{k-u_i}(\y_i^k)\,=:\,R,\quad i=1,\ldots,N\]
then $f^*_k=f^*$ and one may extract $R$ global minimizers.
\end{enumerate}
\end{theorem}

\begin{proof}
The proof of (a) is classical. One first prove that if $(\y_i^k)$ is a nearly optimal solution
of (\ref{hierarchy1}), i.e.
\[f^*_k\leq \sum_{i=1}^N L_{\y_i^k}(p_i)\leq f^*_k+\frac{1}{k},\]
then there exists a subsequence $(k_\ell)$ and a sequence $\y_i$, $i=1,\ldots,N$, such that
\[\lim_{\ell\to\infty}\,y^{k_\ell}_{i\balpha}\,=\,y_{i\balpha},\quad \forall\balpha\in\N^n, i=1,\ldots,N.\]
From this pointwise convergence it easily follows that for every $i=1,\ldots,N$ and $j=1,\ldots,m$,
\[M_k(\y_i)\succeq0,\:M_{k}(g_j\y_i)\succeq0,\quad k=0,1,\ldots\]
By Putinar's theorem \cite[Theorem 2.14]{book} this implies that the sequence
$\y_i$ has a representing measure supported on $\K$, i.e., there exists a finite Borel measure $\mu_i$ on $\K$ such that
\[L_{\y_i}(f)\,=\,\int_\K fd\mu_i,\qquad \forall\, f\in\R[\bx].\]
Moreover, still by pointwise convergence,
\begin{equation}
\label{equalmoments}
L_{\y_i}(q_i\bx^{\balpha})\,=\,\int_\K\bx^{\balpha}q_i(\bx)d\mu_i\,=\,
L_{\y_1}(q_1\bx^{\balpha})\,=\,\int_\K\bx^{\balpha}q_1(\bx)d\mu_1,\quad
\forall\balpha\in\N^n.
\end{equation}
Therefore,
let $d\nu_i:=q_i(\bx)d\mu_i$ which is a probability measure supported on $\K$.
As $\K$ is compact, by (\ref{equalmoments}), $\nu_i=\nu_1$ for every $i=1,\ldots,N$.
Finally, again by pointwise convergence:
\begin{eqnarray*}
f^*\geq\lim_{\ell\to\infty}f^*_{k_\ell}&=&
\lim_{\ell\to\infty}\sum_{i=1}^NL_{\y^{k_\ell}_i}(p_i)
=\sum_{i=1}^N\int_\K p_i d\mu_i\\
&=&\sum_{i=1}^N\int_\K\frac{p_i}{q_i}q_i d\mu_i
=\sum_{i=1}^N\int_\K\frac{p_i}{q_i}d\nu_1\\
&=&\int_\K\left(\sum_{i=1}^N\frac{p_i}{q_i}\right)d\nu_1\geq f^*
\end{eqnarray*}
which proves (a) because $f^*_k$ is monotone non-decreasing. In addition,
$\nu_1$ is an optimal solution of (\ref{lp0}) with optimal value $f^*=\hat{f}$.

Statement (b) follows from the flat extension theorem of Curto and Fialkow
\cite[Theorem 3.7]{book} and each $\y_i$ has an atomic representing measure
supported on $R$ points of $\K$.
\end{proof}

\section{Sparse SDP relaxations}\label{sparse}

In this section we assume that $n$, the number of variables in problem (\ref{defpb}),
is large, say from 10 to 100, and moreover that some sparsity pattern is present
in the polynomial data.

\subsection{GMP formulation}

Let $I_0:=\{1,\ldots,n\}=\cup_{i=1}^NI_i$ with possible overlaps, and 
let $\R[x_k:k\in I_i]$ denote the ring of polynomials in the variables $x_k,\:k\in I_i$.
Denote by $n_i$ the cardinality of $I_i$.

One will assume that $\K\subset\R^n$ in (\ref{setk})  is compact, and one knows
some $M>0$ such that $\bx\in\K\Rightarrow M-\Vert\bx\Vert^2\geq0$.
For every $i\leq N$, introduce the quadratic polynomial $\bx\mapsto g_{m+i}(\bx)=M-\sum_{k\in I_i}x_k^2$.
The index set $\{1,\ldots,m+N\}$ has a partition $\cup_{i=1}^NJ_i$ with $J_i\neq\emptyset$ for every $i=1,\ldots,N$.
In the sequel we assume that for every $i=1,\ldots,N$, $p_i,q_i\in\R[x_k:k\in I_i]$ and
for every $j\in J_i$, $g_j\in\R[x_k:k\in I_i]$.
Next, for every $i=1,\ldots,N$, let
\[\K_i\,:=\,\{\bz\in\R^{n_i}\::\: g_k(\z)\geq0,\:k\in J_i\}\]
so that $\K$ in (\ref{setk}) has the equivalent characterization
\[\K\,=\,\{\bx\in\R^n\::\: (x_k, k\in I_i)\in \K_i,\:i=1,\ldots,N\}.\]
Similarly, for every $i,j\in\{1,\ldots,N\}$ such that
$i\neq j$ and $I_i\cap I_j\neq\emptyset$,
\[\K_{ij}\,=\,\K_{ji}\,:=\,\{(x_k,\,k\in I_i\cap I_j)\::\:(x_k,\,k\in I_i)\in \K_i;\:
(x_k,\,k\in I_j)\in \K_j\:\}.\]

Let $\mathcal{M}(\K)$ be the space of finite Borel measures on
$\K$, and for every $i=1,\ldots,N$, let 
$\pi_i:\def\by{\mathbf{y}}\mathcal{M}(\K)\to
\mathcal{M}(\K_i)$ denote the projection on $\K_i$, that is,
for every $\mu\in\mathcal{M}(\K)$:
\[\pi_i\mu(B)\,:=\,\mu(\{\bx\::\:\bx\in\K;\:(x_k,k\in I_i)\in B\}),\quad\forall B\in
\mathcal{B}(\K_i)\]
where $\mathcal{B}(\K_i)$ is the usual Borel $\sigma$-algebra 
associated with $\K_i$. 

For every $i,j\in\{1,\ldots,N\}$ such that
$i\neq j$ and $I_i\cap I_j\neq\emptyset$, the projection 
$\pi_{ij}:\mathcal{M}(\K_i)\to\mathcal{M}(\K_{ij})$ is also defined in 
an obvious similar manner. For every $i=1,\ldots,N-1$ define the set:
\[U_i\,:=\,\{\,j\in \{i+1,\ldots,N\}\,:\: I_i\cap I_j\neq\emptyset\,\},\]
and consider the infinite dimensional problem
\begin{equation}
\label{lp1}
\begin{array}{rl}
\hat{f}\,:=\,\displaystyle\inf_{\mu_i\in\mathcal{M}(\K_i)}&
\displaystyle\sum_{i=1}^N\int_{\K_i}p_i\,d\mu_i\\
\mbox{s.t.}
&\displaystyle\int_{\K_i}q_i d\mu_i=1,\quad i=1,\ldots,N\\
&\\
&\pi_{ij}(q_i d\mu_i)\,=\,\pi_{ji}(q_jd\mu_j),
\quad \forall j\in U_i,\,i=1,\ldots,N-1.
\end{array}
\end{equation}
\begin{definition}\label{rip}
Sparsity pattern $(I_i)_{i=1}^N$ satisfies the running intersection property if for every $i=2,\ldots,N$:
\[I_i\,\bigcap\,\left(\bigcup_{k=1}^{i-1}I_k\right)\,\subseteq\,I_j,\quad\mbox{ for some }j\leq i-1.\]
\end{definition}

\begin{theorem}
Let $\K\subset\R^n$ in (\ref{setk}) be compact. If the sparsity pattern
$(I_i)_{i=1}^N$ satisfies the running intersection property then
$\hat{f}=f^*$.
\end{theorem}
\begin{proof}
That $\hat{f}\leq f^*$ is straightforward. As $\K$ is compact and $q_i>0$ on $\K$
for every $i=1,\ldots,N$, $f^*=\sum_{i=1}^N f_i(\bx^*)$ for some $\bx\in\K$.
So let $\mu$ be the Dirac measure $\delta_{\bx^*}$ at $\bx^*$ and let
$\nu_i$ be the projection $\pi_i\mu$ of $\mu$ on $\K_i$.
That is $\nu_i=\delta_{(x^*_k,k\in I_i)}$, the Dirac measure at 
the point $(x^*_k,k\in I_i)$ of $\K_i$. Next, for every $i=1,\ldots,N$,
define the measure $d\mu_i:=q_i(\bx^*)^{-1}d\nu_i$. Obviously,
$(\mu_i)$ is a feasible solution of (\ref{lp1}) because
$\mu_i\in\mathcal{M}(\K_i)$ and $\int q_i d\mu_i=1$, for every $i=1,\ldots,N$, and one also has:
\[(x^*_k,\,k\in I_i\cap I_j)\,=\,\pi_{ij}\mu_i=\pi_{ji}\mu_j,\quad\forall j\neq i\:\mbox{such that }I_j\cap I_i\neq\emptyset.\]
Finally, its value satisfies
\[\sum_{i=1}^N\int_{\K_i}p_i d\mu_i\,=\,
\sum_{i=1}^N p_i(\bx^*)/q_i(\bx^*)\,=\,f^*,\]
and so $\hat{f}\leq f^*$.

We next prove the converse inequality $\hat{f}\geq f^*$. 
Let $(\mu_i)$ be an arbitrary feasible solution of 
(\ref{lp1}), and for every $i=1,\ldots,N$, denote by $\nu_i$ the probability measure
on $\K_i$ with density $q_i$ with respect to $\mu_i$, that is,
\[\nu_i(B)\,:=\,\int_{\K_i\cap B}q_i(\bx)\,d\mu_i(\bx),\qquad\forall B\in\mathcal{B}(\K_i).\]
By definition of the linear program (\ref{lp1}), 
$\pi_{ij}\nu_i=\pi_{ji}\nu_j$ for every couple 
$j\neq i$ such that $I_j\cap I_i\neq\emptyset$.
Therefore, by \cite[Lemma B.13]{book} there exists a probability measure $\nu$
on $\K$ such that $\pi_i\nu=\nu_i$ for every $i=1,\ldots,N$.	
But then 
\begin{eqnarray*}
\displaystyle\sum_{i=1}^N\int_{\K_i}p_i\,d\mu_i=
\displaystyle\sum_{i=1}^N\int_{\K_i}\frac{p_i}{q_i}\,d\nu_i
&=&\displaystyle\sum_{i=1}^N\int_{\K_i}\frac{p_i}{q_i}\,d\nu\\
&=&\displaystyle\int_\K\left(\sum_{i=1}^N\frac{p_i}{q_i}\right)\,d\nu\geq f^*
\end{eqnarray*}
and so $\hat{f}\geq f^*$.
\end{proof}

\subsection{A hierarchy of sparse SDP relaxations}

Let $\y=(y_{\balpha})$ be a real sequence indexed in the canonical basis 
$(\bx^{\balpha})$ of $\R[\bx]$. Define the linear functional 
$L_\y:\R[\bx]\to\R$, by:
\[f\:\left(=\sum_{\alpha\in\N^n}f_\alpha\bx^\alpha\right)\:\mapsto\:
\sum_{\alpha\in\N^n}f_\alpha y_\alpha,\qquad\forall f\in\R[\bx].\]
For every $i=1,\ldots,N$, let 
\[\N^{(i)}\,:=\,\{\:\balpha\in\N^n\::\:\alpha_k=0\mbox{ if }k\not\in I_i\:\};
\quad\N^{(i)}_k\,:=\,\{\:\balpha\in\N^{(i)}\::\:\sum_i\alpha_i\leq k\:\}.\]
An obvious similar definition of $\N^{(ij)}$ ($=\N^{(ji)}$)
and $\N^{(ij)}_k$ ($=\N^{(ji)}_k$)
applies when considering $I_j\cap I_i\neq\emptyset$.

Let $\y=(y_{\balpha})$ be a given sequence indexed in the canonical basis of $\R[\bx]$. 
For every $i=1,\ldots,N$, the sparse moment matrix $M_k(\y,I_i)$ associated with $y$, 
has its rows and columns indexed in the canonical basis $(\bx^{\balpha})$ of $\R[x_k\,:\,k\in I_i]$,
and with entries:
\[M_k(\y,I_i)_{\alpha,\beta}=L_\y(\bx^{\alpha+\beta})=y_{\alpha+\beta},\qquad\forall\,\alpha,\beta\in\N^{(i)}_k.\]
Similarly, for a given polynomial $h\in\R[x_k\,:\,k\in I_i]$, the sparse localizing matrix $M_k(h\, \y,\,I_i)$ associated with $\y$ and $h$, has
its rows and columns indexed in the canonical basis $(\bx^{\balpha})$ of $\R[x_k\,:\,k\in I_i]$, and with entries:
\[M_k(h\,\y,I_i)_{\alpha,\beta}=L_\y(h\,\bx^{\alpha+\beta})=
\sum_{\gamma\in\N^{(i)}}h_\gamma y_{\alpha+\beta+\gamma},\qquad\forall\,\alpha,\beta\in\N^{(i)}_k.\]

With $\K\subset\R^n$ defined in (\ref{setk}), let $r_j:=\lceil ({\rm deg}g_j)/2\rceil$, for every $j=1,\ldots,m+N$.
Consider the hierarchy of semidefinite relaxations:
\begin{equation}
\label{primal-sparse}
\begin{array}{rll}
f^*_k=\displaystyle\inf_\y &\displaystyle\sum_{i=1}^N L_\y(p_i)&\\
\mbox{s.t.}& M_k(\y,I_i)\succeq0,&i=1,\ldots,N\\
&M_{k-r_j}(g_j\y,I_i)\succeq0,&\forall j\in J_i,\:i=1,\ldots,N\\
&L_\y(q_i)\,=\,1,&i=1,\ldots,N\\
&L_\y(\bx^{\balpha}q_i)=L_\y(\bx^{\balpha}q_j)=0,&\forall\alpha\in
\N^{(ij)},\,\forall j\in U_i,\:i=1,\ldots,N-1\\
&&\mbox{with }\vert\alpha\vert+\max[{\rm deg}q_i, {\rm deg}q_j]\leq 2k.
\end{array}\end{equation}
\begin{theorem}
\label{th-sparse}
Let $\K\subset\R^n$ in (\ref{setk}) be compact.
Let the sparsity pattern $(I_i)_{i=1}^N$ satisfy the running intersection property,
and consider the hierarchy of semidefinite relaxations defined in (\ref{primal-sparse}). Then:

{\rm (a)}  $f^*_k\uparrow f^*$ as $k\to\infty$.

{\rm (b)} If an optimal solution $\y^*$ of (\ref{primal-sparse}) satisfies
\[{\rm rank}\,\M_k(\y^*,I_i)\,=\,{\rm rank}\,\M_{k-v_i}(\y^*,I_i)\,=:\,R_i,\qquad \forall i=1,\ldots,N,\]
(where $v_i=\max_{j\in J_i}[r_j]$), and
\[{\rm rank}\,\M_k(\y^*,I_i\cap I_j)\,=\,1,\qquad \forall j\in U_i,\:=1,\ldots,N-1,\]
then $f^*_k=f^*$ and one may extract finitely many global minimizers.
\end{theorem}
\begin{proof}
The proof is similar to that of Theorem \ref{th2} and also
to that of \cite[Theorem 4.7]{book}.
One first prove that if $(\y^k)$ is a nearly optimal solution
of (\ref{hierarchy1}), i.e.
\[f^*_k\leq \sum_{i=1}^N L_{\y^k}(p_i)\leq f^*_k+\frac{1}{k},\]
then there exists a subsequence $(k_\ell)$ and a sequence $\y$, such that
\[\lim_{\ell\to\infty}\,y^{k_\ell}_{\balpha}\,=\,y_{\balpha},\quad \forall\balpha\in\N^{(i)}, i=1,\ldots,N.\]
From this pointwise convergence it easily follows that for every $i=1,\ldots,N$ and $j\in J_i$,
\[M_k(\y,I_i)\succeq0,\:M_{k}(g_j\,\y,I_i)\succeq0,\quad j\in J_i;\,i=1,\ldots,N.\]
Now observe that  each set $\K_i\subset\R^{n_i}$  satisfies Assumption \ref{ass1}. Therefore,
by Putinar's theorem \cite[Theorem 2.14]{book} the sequence
$\y^i=(y_{\balpha})$, $\alpha\in\N^{(i)}$ (a subsequence of $\y$), has a representing measure $\mu_i$ supported on $\K_i$.
For every $(i,j)$ with $j\in U_i$, denote by $\y^{ij}$ the sequence $(y_{\balpha})$, $\alpha\in\N^{(ij)}$.
Again, by pointwise convergence, $L_\y(q_i)=1$, $i=1,\ldots,N$, and
\begin{equation}
\label{equalmoments-new}
L_{\y}(q_i\bx^{\balpha})\,=\,\int_{\K_i}\bx^{\balpha}q_j(\bx)d\mu_j\,=\,
\int_{\K_j}\bx^{\balpha}q_j(\bx)d\mu_j,\quad\forall\balpha\in\N^{(ij)},\:\forall j\in U_i.
\end{equation}
Therefore, for every $i=1,\ldots,N$, $d\nu_i:=q_i(\bx)d\mu_i$
is a finite Borel probability measure supported on $\K_i$. As measures on compact sets are moment determinate, 
(\ref{equalmoments-new}) yields:
\[\pi_{ij}\nu_i\,=\,\pi_{ji}\nu_j,\qquad \forall (i,j),\,j\in U_i. \]
Therefore, by \cite[Lemma B.13]{book} there exists a probability measure $\nu$
on $\K$ such that $\pi_i\nu=\nu_i$ for every $i=1,\ldots,N$.	
But then 
\begin{eqnarray*}
f^*\geq\lim_{\ell\to\infty}f^*_{k_\ell}&=&
\lim_{\ell\to\infty}\sum_{i=1}^NL_{\y^{k_\ell}_i}(p_i)
=\sum_{i=1}^N\int_\K p_i d\mu_i\\
&=&\sum_{i=1}^N\int_\K\frac{p_i}{q_i}q_i d\mu_i
=\sum_{i=1}^N\int_\K\frac{p_i}{q_i}d\nu_i\\
&=&\int_\K\left(\sum_{i=1}^N\frac{p_i}{q_i}\right)d\nu\geq f^*.
\end{eqnarray*}
As the converging subsequence was arbitrary, and $(f^*_k)$ is monotone non decreasing,
we finally get $f^*_k\uparrow f^*$. In addition,
$\nu$ is an optimal solution of (\ref{lp1}) with optimal value $f^*=\hat{f}$.

The proof of (b) is as in \cite{sparse}
and uses the flat extension theorem of Curto and Fialkow
\cite[Theorem 3.7]{book} from which, the sequence $\y^*_i=(y^*_\alpha)$, $\alpha\in\N^{(i)}$, has an atomic representing measure
supported on $R_i$ points of $\K_i$, for every $i=1,\ldots,N$.
\end{proof}

\section{GloptiPoly and examples}\label{gmp}

In this section we show that the generalized moment problem (GMP)
formulation of rational optimization problem (\ref{defpb})
has a straightforward Matlab implementation when using our
software GloptiPoly 3 \cite{g3}. Rather than explaining the
approach in full generality with awkward notations, we describe
three simple examples. 

\subsection{Wilkinson-like rational function}\label{wilkinson}

Consider the elementary univariate rational optimization problem 
\[
f^* = \sup_{x \in \R} f(x), \quad f(x) =  \sum_{i=1}^N \frac{p_i(x)}{q_i(x)} 
= \sum_{i=1}^N \frac{1}{x^2+i}
\]
with $N$ an integer.
The only real critical point is $x=0$, at which the objective
function takes its maximum
\[
f^* = f(0) = \sum_{i=1}^N \frac{1}{i}.
\]
Reducing to the same denominator
\[
f(x) = \frac{\sum_i\prod_{j\neq i}(x^2+j)}{\prod_i(x^2+i)} = \frac{p(x^2)}{q(x^2)}
\]
yields the well-known Wilkinson polynomial $q$ whose squared root moduli
are the integers from $1$ to $N$. This polynomial was described in the mid 1960s
by James H. Wilkinson to illustrate the difficulty of finding numerically
the roots of polynomials. If we choose e.g. $N=20$, reduction
to the same denominator is hopeless since the constant coefficient in monic
polynomial $q$
is $20!=2432902008176640000$. The GMP formulation (\ref{lp0}) of this problem
reads (up to replacing $\inf$ with $\sup$ in the objective function):
\[
\begin{array}{ll}
\sup_{\mu_i \in {\mathcal M}(\R)} & \sum_{i=1}^N \int_\R p_i d\mu_i \\
\mathrm{s.t.} & \int_\R q_1 d\mu_1 = 1 \\
& \int_\R x^{\alpha} q_i d\mu_i = \int_\R x^{\alpha} q_1 d\mu_1, \quad \forall \alpha \in {\mathbb N}^n, \: i=1,\ldots,N.
\end{array}
\]
Our Matlab script to model and solve this problem is as follows:
\begin{verbatim}
N = 20; mpol('x',N); % create variables
q = cell(N,1); % problem data
mu = cell(N,1); % measures
for i = 1:N, q{i} = i+x(i)^2; mu{i} = meas(x(i)); end
% model GMP
k = 0; % relaxation order
f = mass(mu{1}); % objective function
e = [mom(q{1}) == 1]; % moment contraints
for i = 2:N
 f = f + mass(mu{i});
 e = [e; mom(mmon(x(1),k)*q{1}) == mom(mmon(x(i),k)*q{i})];
end
% model SDP relaxation of GMP
P = msdp(max(f),e);
% solve SDP relaxation
[stat,obj] = msol(P)
\end{verbatim}
Instructions {\tt mpol}, {\tt meas},
{\tt mass}, {\tt mom}, {\tt mmon}, {\tt msdp}, {\tt max}
and {\tt msol} are GloptiPoly 3 commands, see the user's
guide \cite{g3} for more information. For readers who are not
familiar with this package, variable {\tt f}
is the objective function to be maximized. Since $p_i=1$ for all $i=1,\ldots,N$,
it is the sum of masses of measures $\mu_i$.
Vector {\tt e} stores the linear moment constraints and
the instruction {\tt mmon(x,k)} generates all monomials
of variable {\tt x} up to degree {\tt k}.
Finally, instruction {\tt msdp} generates the SDP relaxation
of the GMP, and {\tt msol} solves the SDP problem with
the default conic solver (SeDuMi 1.3 in our case).

At the first SDP relaxation (i.e. for {\tt k=0})
we obtain a rank-one moment matrix corresponding to a Dirac at $x^*=0$:
\begin{verbatim}
>> [stat,obj] = msol(P)
Global optimality certified numerically
stat =
     1
obj =
    3.5977
\end{verbatim}
which is consistent with Maple's
\begin{verbatim}
> f := sum(1/(x^2+i), i=1..20);
> evalf(subs(x = 0, f));
    3.5977
\end{verbatim}

Note that for this example Assumption \ref{ass1} is violated,
since we optimize over the non-bounded set $\K = {\mathbb R}$.
In spite of this, we could solve the problem globally.

\subsection{Relevance of the compactness assumption}\label{sec-bounds}

With this elementary example we would like to emphasize the
practical relevance of Assumption \ref{ass1} on the existence of an algebraic
certificate of compactness of set $\K$.
Consider the univariate problem
\begin{equation}\label{pb-bounds}
f^* = \inf_{x \in \K} f(x), \quad f(x) = \frac{1+x+x^2}{1+x^2} + \frac{1+x^2}{1+2x^2}.
\end{equation}
First let $\K = {\mathbb R}$.
The numerator of the gradient of $f(x)$ has two real roots,
one of which being the global minimum located at $x^*=-1.4215$ for which
$f^*=1.1286$. The following GloptiPoly script models and solves the SDP
relaxations of orders $k=0,\ldots,9$ of the GMP formulation of this problem:
\begin{verbatim}
mpol x1 x2
f1 = 1+x1+x1^2; g1 = 1+x1^2; f2 = 1+x2^2; g2 = 1+2*x2^2;
mu1 = meas(x1); mu2 = meas(x2);
bounds = [];
for k = 0:9
 P = msdp(min(mom(f1)+mom(f2)), ...
     mom(mmon(x1,k)*g1) == mom(mmon(x2,k)*g2), mom(g1) == 1);
 [stat, obj] = msol(P);
 bounds = [bounds; obj];
end
bounds
\end{verbatim}
In vector {\tt bounds} we retrieve the following monotically increasing
sequence of lower bounds $f^*_k$ (up to 5 digits)
obtained by solving the SDP relaxations (\ref{hierarchy1}):
\begin{table}[h!]
\begin{center}
\begin{tabular}{c|c||c|c}
order $k$ & bound $f^*_k$ & order $k$ & bound $f^*_k$\\ \hline
0 & 1.0000 & 5 & 1.0793\\
1 & 1.0000 & 6 & 1.1264\\
2 & 1.0170 & 7 & 1.1283\\
3 & 1.0220 & 8 & 1.1286\\
4 & 1.0633 & 9 & 1.1286\\
\end{tabular}
\caption{Lower bounds for SDP relaxations of problem (\ref{pb-bounds}).\label{ex-bounds}}
\end{center}
\end{table}
At SDP relaxation $k=9$, GloptiPoly certifies global optimality and
extracts the global minimizer.
Table \ref{ex-bounds} shows that the convergence of the hierarchy of
SDP relaxations is rather slow for this very simple example.
This is due to the fact that Assumption \ref{ass1}
is violated, since we optimize over the non-bounded set $\K = {\mathbb R}$.
\begin{figure}[h!]
\begin{center}
\includegraphics[width=0.7\textwidth]{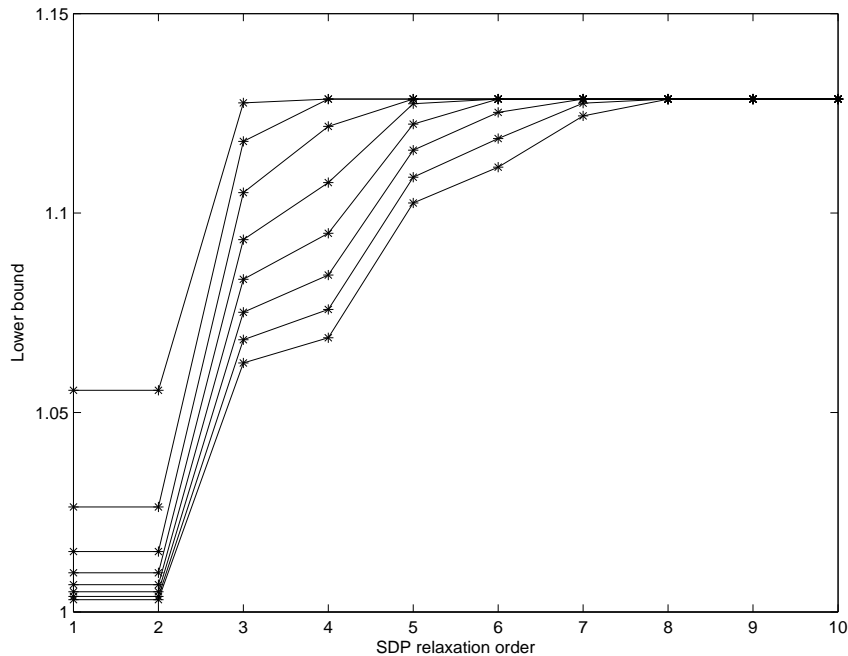}
\caption{Lower bounds for SDP relaxations of problem (\ref{pb-bounds})
on bounded sets $K=[-R,\: R]$ for $R=2$ (top curve) to $R=9$ (bottom curve).
\label{fig-bounds}}
\end{center}
\end{figure}

On Figure \ref{fig-bounds} we report the sequences of lower bounds
obtained by solving the SDP relaxations of problem (\ref{pb-bounds})
on compact sets $\K = [-R, \: R]$ for $R=2,3,\ldots,9$.

\subsection{Exploiting sparsity with GloptiPoly}\label{sparse_ex}

Even though version 3 of GloptiPoly is designed to exploit problem sparsity,
there is no illustration of this feature in the software user's guide \cite{g3}.
In this section we provide such a simple example. Note also that GloptiPoly is not
able to detect sparsity in a given problem, contrary to SparsePOP which uses
a heuristic to find chordal extensions of graphs \cite{sparsepop}.
However, SparsePOP is not designed to handle directly rational optimization problems.

Consider the elementary example of \cite[Section 3.2]{sparse}:
\[
\begin{array}{ll}
\inf_{\bx \in \R^4} & x_1x_2+x_1x_3+x_1x_4 \\
\mathrm{s.t.} & x^2_1+x^2_2 \leq 1 \\
& x^2_1+x^2_3 \leq 2 \\
& x^2_1+x^2_4 \leq 3
\end{array}
\]
for which the variable index subsets $I_1=\{1,2\}$, $I_2=\{1,3\}$, $I_3=\{1,4\}$
satisfy the running intersection property of Definition \ref{rip}.
Note that this problem is a particular case of (\ref{defpb})
with a polynomial objective function.

Without exploiting sparsity, the GloptiPoly script to solve
this problem is as follows:
\begin{verbatim}
mpol x1 x2 x3 x4
Pdense = msdp(min(x1*x2+x1*x3+x1*x4), ...
         x1^2+x2^2<=1,x1^2+x3^2<=2,x1^2+x4^2<=3,2);
[stat,obj] = msol(Pdense);
\end{verbatim}
GloptiPoly certifies global optimality
with a moment matrix of size 15, and 3 localizing matrices of size 5.
And here is the script exploiting sparsity, splitting the variables
into several measures $\mu_i$ consistently with subsets $I_i$:
\begin{verbatim}
mpol x1 3
mpol x2 x3 x4
mu(1) = meas([x1(1) x2]); % first measure on x1 and x2
mu(2) = meas([x1(2) x3]); % second measure on x1 and x3
mu(3) = meas([x1(3) x4]); % third measure on x1 and x4
f = mom(x1(1)*x2)+mom(x1(2)*x3)+mom(x1(3)*x4); % objective function
k = 3; % SDP relaxation order
m1 = mom(mmon(x1(1),k)); % moments of first measure
m2 = mom(mmon(x1(2),k)); % moments of second measure
m3 = mom(mmon(x1(3),k)); % moments of third measure
K = [x1(1)^2+x2^2<=1, x1(2)^2+x3^2<=2, x1(3)^2+x4^2<=3]; % supports
Psparse = msdp(min(f),m1==m2,m3==m2,K,mass(mu)==1);
[stat,obj] = msol(Psparse);
\end{verbatim}
GloptiPoly certifies global optimality
with 3 moment matrices of size 6, and 3 localizing matrices of size 3.

\subsection{Comparison with the epigraph approach}

In most of the examples we have processed, the epigraph approach
described in the Introduction (consisting of introducing one lifting
variable for each rational term in the objective function) was
less efficient than the GMP approach. Typically, the order
of the SDP relaxation (and hence its size) required to certify
global optimality is typically larger with the epigraph approach.

When evaluating the epigraph approach,
we also observed that it is numerically preferable to replace
the inequality constraints $r_i q_i(\bx) - p_i(\bx) \geq 0$
with equality constraints $r_i q_i(\bx) - p_i(\bx) = 0$
in the definition of semi-algebraic set $\hat{\K}$ in
(\ref{defpb1}). For the example of Section \ref{wilkinson}
the epigraph approach with inequalities certifies global optimality
at order $k=5$, whereas the epigraph approach with equalities
requires $k=1$.

As a typical illustration of the issues faced with the epigraph approach
consider the example with eighth-degree terms
\begin{equation}\label{rat_ex}
\inf_{\bx \in \R^2} f(x) = \sum_{i=1}^{10}
\frac{(x_1+x_2)(x^2_1+x^2_1x^2_2+x^4_2+i^2)-(ix^2_2+1)(x^4_1+x^2_2+2i)}
{(x^4_1+x^2_2+2i)(x^2_1+x^2_1x^2_2+x^4_2+i^2)}
\end{equation}
which is cooked up to have several local optima and sufficiently high
degree to prevent reduction to the same denominator.
After a suitable scaling to make critical points fit within the
box $[-1,1]^2$, as required by the moment SDP relaxations formulated
in the power basis \cite[Section 6.5]{g2},
the GMP approach yields a certificate of global optimality
with $x^*_1=-0.60450$, $x^*_2=-2.2045$, $f^*=-6.2844$
at order $k=6$ in a few seconds on a standard PC.
In contrast, the epigraph approach does not provide a certificate
for an order as high as $k=10$, requiring more than one minute of
CPU time.

\subsection{Shekel's foxholes}

Consider the modified Shekel foxholes rational function minimization
problem \cite{bersini}
\begin{equation}\label{shekel}
\min_{x \in {\mathbb R}^n} \sum_{i=1}^N \frac{1}{\sum_{j=1}^n (x_j-a_{ij})^2+c_i}
\end{equation}
whose data $a_{ij}$, $c_i$, $i=1,\ldots,N$, $j=1,\ldots,n$
can be found in \cite[Table 16]{ali}. This function is designed to have many
local minima, and we know the global minimum in the case $n=5$, see
\cite[Table 17]{ali}.
After a suitable scaling to make critical points fit within the
box $[0,1]^5$, and after addition of a Euclidean ball constraint
centered in the box,
the GMP approach yields a certificate
of global optimality at order $k=3$ in less than one minute on a
standard PC. The extracted minimizer is $x^*_1=8.0254$, $x^*_2=9.1483$,
$x^*_3=5.1138$, $x^*_4=7.6213$, $x^*_5=4.5638$, which matches with
the known global minimizer to four significant digits. This point
can be refined if given an initial guess for a local optimization
method. If we use a standard quasi-Newton BFGS algorithm, we obtain
after a few iterations a point matching the known global
minimizer to eight significant digits.

In the case $n=10$, for which the global minimum is given in
\cite[Table 17]{ali}, the GMP approach yields a certificate
of global optimality at order $k=2$ in about 750 seconds of
CPU time. Here too, we observe that the extracted minimizer
$x^*_1=8.0249$, $x^*_2=9.1518$, $x^*_3=5.1140$, $x^*_4=7.6209$,
$x^*_5=4.5640$, $x^*_6=4.7110$, $x^*_7=2.996$, $x^*_8=6.1259$,
$x^*_9=0.73424$, $x^*_10=4.9820$ is a good approximation
to the minimizer, with four correct significant digits.
If necessary, this point can be used as an initial guess
for refining with a local solver.

Note that it is not possible to exploit problem sparsity
in this case, since all the variables appear in each
term in sum (\ref{shekel}).

\subsection{Rosenbrock's function}

Consider the rational optimization problem
\begin{equation}\label{rosenbrock}
f^* =  \max_{x \in {\mathbb R}^n} \sum_{i=1}^{n-1} \frac{1}{100(x_{i+1}-x^2_i)^2+(x_i-1)^2+1}
\end{equation}
which has the same critical points as the well-known Rosenbrock problem
\[
\min_{x \in {\mathbb R}^n} \sum_{i=1}^{n-1} (100(x_{i+1}-x^2_i)^2+(x_i-1)^2)
\]
whose geometry is troublesome for local optimization solvers.
It can been easily shown that the global maximum $f^*=1$
of problem (\ref{rosenbrock}) is achieved at $x^*_i = 1$, $i=1,\ldots,n$.
Our experiments with local optimization algorithms reveal
that standard quasi-Newton solvers or functions of the Optimization
toolbox for Matlab, called repeatedly with random initial guesses,
typically yield local maxima quite far from the global maximum.

With our GMP approach, after exploiting sparsity and
adding bound constraints $x^2_i \leq 16$,
$i=1,\ldots,n$, we could solve problem (\ref{rosenbrock})
with a certificate of global optimality for $n$ up to $1000$.
Typical CPU times range
from 10 seconds for $n=100$ to 500 seconds for $n=1000$.

\section{Conclusion}

The problem of minimizing the sum of {\it many} low-degree (typically non-convex) 
rational fractions on a (typically non-convex) semi-algebraic set
arises in several important applications, and notably in 
computer vision (triangulation, estimation of the fundamental matrix
in epipolar geometry) and in systems control ($H_2$ optimal control
with a fixed-order controller of a linear system subject to parametric uncertainty).
These engineering problems motivated our work, but the application of our techniques
to computer vision and systems control will be described elsewhere.
These fractional programming problems being non convex, local optimization approaches 
yield only upper bounds on the optimum.

In this paper we were interested in computing the global minimum (and possibly 
global minimizers) or at least, computing valid lower bounds on the global minimum, for fractional programs 
involving a sum with many terms.
We have used a semidefinite programming (SDP) relaxation approach by formulating
the rational optimization problem as an instance of the generalized moment problem (GMP).
In addition, problem structure can be sometimes exploited in the case
where the number of variables is large but sparsity is present.
Numerical experiments with our public-domain software GloptiPoly interfaced
with off-the-shelf semidefinite programming solvers indicate that the approach can solve
problems that can be challenging for state-of-the-art
global optimization algorithms. This is consistent with the experiments
made in \cite{iccv} where the (dense) SDP relaxation approach was
first applied to (polynomial) optimization problems of computer vision.

For larger and/or ill-conditioned problems, it can happen that GloptiPoly
extracts from the moment matrix a minimizer which is not very accurate.
It can also happen that GloptiPoly is not able to extract a minimizer,
in which case first-order moments approximate the minimizer (provided
it is unique, which is generically true for rational optimization).
The approximate minimizer can be then input to any local optimization
algorithm as an initial guess.

A comparison of our approach with other techniques of global optimization
(reported e.g. on Hans Mittelmann's or Arnold Neumaier's webpages)
is out of the scope of this paper.
We believe however that such a comparison would be fair only if no expert
tuning is required for alternative algorithms. Indeed, when using
GloptiPoly the only assumption we make is that we know a ball
containing the global optimizer. Besides this, our results are fully
reproducible (Matlab files reproducing our examples are available
upon request) and our SDP relaxations are solved with general-purpose
semidefinite programming solvers.

\section*{Acknowledgments}

We are grateful to Michel Devy, Jean-Jos\'e Orteu, Tom\'a\v s Pajdla,
Thierry Sentenac and Rekha Thomas for insightful discussions
on applications of real algebraic geometry and SDP in computer vision,
and to Josh Taylor for his feedback on the example of section \ref{sparse_ex}.


\begin{thebibliography}{xx}
\bibitem{ali}
M. M. Ali, C. Khompatraporn, Z. B. Zabinsky.
A numerical evaluation of several stochastic algorithms on selected continuous
global optimization test problems.
{\em J. Global Optim.}, 31(4):635-672, 2005.
\bibitem{bersini}
H. Bersini, M. Dorigo, S. Langerman, G. Seront, L. Gambardella.
Results of the first international contest on evolutionary optimisation.
{\em IEEE Intl. Conf. Evolutionary Computation}, Nagoya, Japan, 1996.
\bibitem{iccv}
F. Kahl, D. Henrion. Globally optimal estimates for geometric
reconstruction problems. {\em IEEE Intl. Conf. Computer Vision},
Beijing, China, 2005.
\bibitem{g2}
D. Henrion, J. B. Lasserre. GloptiPoly: global optimization over
polynomials with Matlab and SeDuMi. {\em ACM Trans. Math. Software},
29:165--194, 2003.
\bibitem{g3}
D. Henrion, J. B. Lasserre, J. L\"ofberg. GloptiPoly 3: moments, 
optimization and semidefinite programming. {\em Optim. Methods and Software},
24:761--779, 2009. 
\bibitem{deklerk}
D. Jibetean, E. de Klerk. Global optimization of rational functions:
a semidefinite programming approach. {\em Math. Programming}, 106:93-109, 2006. 
\bibitem{sparse}
J. B. Lasserre. Convergent SDP relaxations in polynomial optimization
with sparsity. {\em SIAM J. Optim.}, 17:822-843, 2006.
\bibitem{gmp}
J. B. Lasserre. A semidefinite programming approach to the generalized
problem of moments. {\em Math. Programming}, 112:65-92, 2008.
\bibitem{book}
J. B. Lasserre. {\em Moments, positive polynomials and their applications},
Imperial College Press, London, 2009.
\bibitem{schaible}
S. Schaible, J. Shi. Fractional programming: the sum-of-ratios case.
{\em Optim. Methods Software}, 18(2):219--229, 2003.
\bibitem{waki}
S. Waki, S. Kim, M. Kojima, M. Muramatsu. Sums of squares and
semidefinite programming relaxations for polynomial optimization problems
with structured sparsity. {\em SIAM J. Optim.}, 17:218-242, 2006.
\bibitem{sparsepop}
H. Waki, S. Kim, M. Kojima, M. Muramatsu, H. Sugimoto. SparsePOP:
a sparse semidefinite programming relaxation of polynomial
optimization problems. {\em ACM Trans. Math. Software}, 35:1-13, 2008.
\end{thebibliography}
\end{document}